\documentclass{amsart2010}
\usepackage{graphicx}
\usepackage{amsmath}
\usepackage{amsfonts}
\usepackage{amssymb}

\newtheorem{thm}{Theorem}
\newtheorem{cor}[thm]{Corollary}
\newtheorem{lem}[thm]{Lemma}
\newtheorem{prop}[thm]{Proposition}
\theoremstyle{definition}
\newtheorem{defn}[thm]{Definition}
\newtheorem{rem}[thm]{Remark}
\newtheorem*{theorem*}{Theorem}

\renewcommand{\qedsymbol}{$\blacksquare$}
\usepackage{float}
\usepackage{tikz}
\usetikzlibrary{shapes,shapes.geometric,arrows,fit,calc,positioning,automata}

\theoremstyle{remark}

\numberwithin{equation}{section}

\def\<{\langle}
\def\>{\rangle}
\begin{document}
\title[]{Minimal invariant subspaces for an affine composition operator}
\author{João R. Carmo, Ben Hur Eidt and  S. Waleed Noor}%
\address{IMECC, Universidade Estadual de Campinas, Campinas-SP, Brazil.}
\email{$\mathrm{joao.mr2@hotmail.com}$ (João R. Carmo)}
\email{$\mathrm{b264387@dac.unicamp.br}$ (Ben Hur Eidt)} 
\email{$\mathrm{waleed@unicamp.br}$ (Waleed Noor) Corresponding Author}

\begin{abstract}  The composition operator $C_{\phi_a}f=f\circ\phi_a$ on the Hardy-Hilbert space $H^2(\mathbb{D})$ with affine symbol $\phi_a(z)=az+1-a$ and $0<a<1$ has the property that the Invariant Subspace Problem for complex separable Hilbert spaces holds if and only if every minimal invariant subspace for $C_{\phi_a}$ is one-dimensional. These minimal invariant subspaces are always singly-generated $ K_f := \overline{\mathrm{span} \{f, C_{\phi_a}f, C^2_{\phi_a}f, \ldots \}}$ for some $f\in H^2(\mathbb{D})$. In this article we characterize the minimal $K_f$ when $f$ has a nonzero limit at the point $1$ or if its derivative $f'$ is bounded near $1$. We also consider the role of the zero set of $f$ in determining $K_f$. Finally we prove a result linking universality in the sense of Rota with cyclicity.
\end{abstract}

{\subjclass[2010]{Primary; Secondary}}
\maketitle{}

\section{Introduction}
The \textit{Invariant Subspace Problem} (ISP) is one of the major open problems in operator theory. Let $H$ be a complex separable Hilbert space. If $T \in B(H)$ (a bounded linear operator on $H$), does $T$ have a non-trivial invariant subspace? 

By a non-trivial invariant subspace we mean a closed subspace $M \subseteq H$ such that $M \neq \{0\}$, $M \neq H$ and $T(M) \subseteq M$. We note that the only open case is when $H$ is separable. There are several possible approaches to the ISP and the monograph of Partington and Chalendar \cite{modernisp} describe some of the modern ones. One such approach is based on the concept of a \textit{universal operator} introduced by Rota \cite{rota}.

\begin{defn}
	Let $H$ be a complex, separable and infinite dimensional Hilbert space. An operator $U \in B(H)$ is said to be \textbf{universal} for $H$ if for any $T \in B(H)$ there exists an invariant subspace $M$ of $U$,  a complex number $\alpha$  and an isomorphism $S: M \to H$  such that $\alpha T = S U_{|M} S^{-1}$, i.e. $U$ restricted to $M$ is conjugate to a multiple of $T$.
\end{defn}

There exists a connection between universal operators and the ISP. The ISP  holds true for $H$ if and only if each \emph{minimal} invariant subspace $M$ of $U$ is one-dimensional. Here, the minimality of $M$ implies that it contains no proper $U$-invariant subspace. Until recently the main method for identifying a universal operator has been the Caradus criterion (see \cite{caradus}). More recently Pozzi \cite{pozzi} generalized this classical result and obtained the following theorem that we will call the \textit{Caradus-Pozzi criterion}.

\begin{thm}\label{Caradus-Pozzi}
	An operator  $U \in B(H)$ is universal if it satisfies:
	\begin{enumerate}
		\item $Ker \ U$ is infinite-dimensional.
		\item $U(H)$ has finite codimension (and is hence closed).
	\end{enumerate}
	
\end{thm}
In this paper, we will deal with the Hardy-Hilbert space $H^2:=H^2(\mathbb{D})$ of  holomorphic functions $f : \mathbb{D} \to \mathbb{C}$ on the open unit disk $\mathbb{D}$ that satisfy
$$\|f\|_{H^2}^2 = \sup\limits_{0 < r < 1} \frac{1}{2\pi} \int\limits_{0}^{2\pi} |f(re^{i \theta} )|^2 d\theta < \infty.$$
Let $H^\infty$ denote the space bounded holomorphic functions on $\mathbb{D}$. If $\phi:\mathbb{D} \to \mathbb{D}$ is a holomorphic self-map of $\mathbb{D}$, then the \textit{composition operator} with \textit{symbol} $\phi$ is defined by $C_{\phi}(f) = f \circ \phi$. By Littlewood's subordination theorem $C_{\phi}$ is always bounded on $H^2$. The study of composition operators centers around the interaction between the function-theoretic properties of the symbol $\phi$ and the operator-theoretic properties of $C_\phi$. 

Let $LFT(\mathbb{D})$ denote the set of all linear fractional self-maps of $\mathbb{D}$. If $\phi \in LFT(\mathbb{D})$, we say that $\phi$ is \textit{hyperbolic} if it
has two distinct fixed points outside $\mathbb{D}$.
Due to the remarkable work of Nordgren, Rosenthal and Wintrobe (see \cite{Nordgren-Rosenthal-Wintrobe}) we know that if $\phi \in LFT(\mathbb{D})$ is a \textit{hyperbolic automorphism} (both fixed points on the unit circle $S^1$) then $C_{\phi} - \lambda$ is universal for every $\lambda$ in the interior of the spectra of $\phi$. Recently, this result was extended by Carmo and Noor \cite[Theorem 3.1]{noor} to non-automorphic hyperbolic self-maps.

\begin{thm}\label{xxx}
	Let $\phi \in LFT(\mathbb{D})$. Then $C_{\phi} - \lambda$ is universal on $H^2(\mathbb{D})$ for some $\lambda \in \mathbb{C}$ if and only if $\phi$ is hyperbolic.
\end{thm}

The hyperbolic automorphism induced $C_\Phi$ has been studied extensively during the last thirty years where
$$\Phi(z) = \frac{z + b}{bz + 1}$$
with $0 < b < 1$ and having fixed points at $1$ and $-1$ (see \cite{Chkliar},\cite{Gallardo-Gorkin},\cite{Matache minimal},\cite{Matache eigenfunction},\cite{Mortini}). Our focus here is on the non-automorphic case where $\phi$ has a fixed point in $S^1$ and another one outside $\overline{\mathbb{D}}$ (possibly at $\infty$). In that case, Hurst \cite[Theorem 8]{Hurst} proved  that $C_{\phi}$ is similar to $C_{\phi_a}$ where 
\[
\phi_a(z) = az + 1 - a, \ \ \  a\in (0,1)
\] 
and $1,\infty$ are the fixed point of $\phi_a$. By Theorem \ref{xxx} one can approach the ISP using the symbol $\phi_a$. Noting that $\mathrm{Lat}(C_{\phi_a}) = \mathrm{Lat}(C_{\phi_a} - \lambda)$ for any $\lambda\in\mathbb{C}$, where $\mathrm{Lat}(T)$ is the collection of all invariant subspaces of an operator $T$, the following was shown in \cite{noor}.

\begin{thm}
	For any $a \in (0,1)$, the $ISP$ has a positive solution if and only if every minimal invariant subspace of $C_{\phi_a}$ has dimension 1. 
\end{thm}
If $M$ is a minimal invariant subspace of $C_{\phi_a}$, then we necessarily have $$M = K_f := \overline{span \{f, C_{\phi_a}f, C^2_{\phi_a}f, \ldots \}}$$ 
for every nonzero $f \in M$. So to study the minimal invariant subspaces we need to understand the so-called \emph{cyclic} subspaces $K_f$ for $f \in H^2$. In fact, we need only consider $f$ that are analytic across each point of $S^1$ except $1$ (see \cite[Prop. 5.1]{noor}).

The plan of the paper is the following. In Section 2 we will introduce some preliminary results and definitions. In particular the Nevanlinna counting function and the notion of an \emph{eventually bounded function} (simply $\mathbf{EB}$) are introduced. When $f\in H^2$ is $\mathbf{EB}$, then $K_f$ always contains an $H^\infty$ function (Proposition \ref{bounded element}). Curiously functions with $\mathbf{EB}$ derivatives can never be hypercyclic vectors for $C_{\phi_a}$ (see  Proposition \ref{hypercyclic vector}). In Section 3, we prove the main results regarding the minimal invariant subspaces of $C_{\phi_a}$. In particular, we prove the veracity of the equivalence
$$K_f \textrm{ is minimal }\Longleftrightarrow \mathrm{dim} \ K_f = 1.$$
for a variety of classes of $f\in H^2$ including those with non-zero boundary limits at $1$ (Theorem \ref{limit1}), when $f'$ is $\mathbf{EB}$ and $(f(1-a^n))_{n\in\mathbb{N}}$ is bounded away from $0$ (Corollary \ref{CORAC}), and when $f$ is analytic at $1$ (Theorem \ref{analyticat1}). In Section 4 we prove results that connect the zero set of $f$ with properties of $K_f$. Finally in Section 5 we provide a sufficient condition for a cyclic operator on a Hilbert space to be universal (see Theorem \ref{newcar}). This is relevant since the best known examples of universal operators are similar to coanalytic Toeplitz operators, and these are always cyclic (see \cite{Wogen}).

\section{Preliminaries}
\subsection{The Nevannlina counting function} In the Shapiro's seminal work \cite{Shapiro Annals} the essential norm of a composition operator was determined using the Nevannlina counting function. It will play an important role here also. For a holomorphic map $\phi: \mathbb{D} \to \mathbb{D}$ and each $w \in \mathbb{D}\setminus \{\phi(0) \}$, define 
$$N_{\phi}(w) = \begin{cases}
	\sum\limits_{z \in \phi^{-1}\{w\}} \log \frac{1}{|z|}  \,\    \textrm{ if  } w \in \phi(\mathbb{D}) \setminus\{\phi(0)\}. \\
	\ \ \ \ 0 \,\ \ \ \ \ \ \ \ \ \ \ \ \ \textrm{ if  }w \notin \phi(\mathbb{D}).
\end{cases}
$$
We shall need the following two results.

\begin{thm}[\cite{Shapiro Text}, section 10.1 ]\label{NVF}
	If $\phi: \mathbb{D} \to \mathbb{D}$ is analytic then $\forall f \in H^2(\mathbb{D})$ we have
	
	$$\|C_{\phi}f\|_{H^2}^2= 2 \int\limits_{\mathbb{D}}|f'(w)|^2 N_{\phi}(w)dA(w) + |f(\phi(0))|^2.$$
\end{thm}

\begin{thm}[\cite{Shapiro Text}, section 10.3]\label{CVF} 
	If $f$ is a non-negative measurable function in $\mathbb{D}$ and $\phi$ is a holomorphic self-map of $\mathbb{D}$ then  $$\begin{aligned}
		\int\limits_{\mathbb{D}} f(w)N_{\phi}(w)dA(w) & =  \int\limits_{\mathbb{D}} f(\phi(z))|\phi'(z)|^2 \log{\frac{1}{|z|}}dA(z) 
	\end{aligned}$$
\end{thm}

We will specialize these formulas for our symbol $\phi_a$ in the following sections. 

\subsection{Eventually bounded functions} For each $a \in (0,1)$ and $n \in \mathbb{N}$, note that $\underbrace{\phi_a \circ \ldots \circ \phi_a}_{n \textrm{ times }} = \phi_{a^n}$ and hence 
$$C_{\phi_{a^n}} = C_{\phi_a}^{n}.$$
For each $n \in \mathbb{N}$, define the disk $D_n := \phi_{a^n}(\mathbb{D}) = a^n\mathbb{D}+1-a^n$ with center $1-a^n$ and radius $a^n$ in $\mathbb{C}$. We call a holomorphic function $g$ defined on $\mathbb{D}$ \textit{eventually bounded at $1$} (or simply $\mathbf{EB}$) if $g$ is bounded on $D_{n}$ for some $n \in \mathbb{N}$ (and consequently bounded on $D_m$ for each $m \geq n$). Note that every $H^\infty$ function is trivially $\mathbf{EB}$. In order to employ Theorems \ref{NVF} and \ref{CVF}, we will need $f'$ to be $\mathbf{EB}$. This in fact implies that $f$ is also $\mathbf{EB}$ since if $f'$ is bounded on some $D_n$, then
\[
|f(z)|\leq|z-z_0|\sup_{w\in D_n}|f'(w)|+|f(z_0)|
\]
for any $z\in D_n$ and some fixed $z_0\in D_n$. 
\begin{center}
	\begin{tikzpicture}
		\draw[dotted] (0,0) circle [radius=3cm];
		\draw[dotted] (1.5, 0) circle [radius=1.5cm];
		\draw[dotted](2.25, 0) circle [radius=0.75cm];
		\draw[dotted](2.675, 0) circle [radius= 0.325cm];
		\draw[->] (3,-3.5) -- (3,3.5);
		\node[label={right: $\{z \,\ | \,\ Re(z) = 1 \}$} ] at (3 ,3) {};
		\node[label={left: $\mathbb{D}$} ] at (-3,0) {}; 
		\node[label={left: $D_1$} ] at (0,0) {};
		\node[label={left: $D_2$} ] at (1.5,0) {};
		\node[label={left: $D_3$} ] at (2.5 ,0) {};
		\node[label={right: $1$} ] at (3 ,0) {};
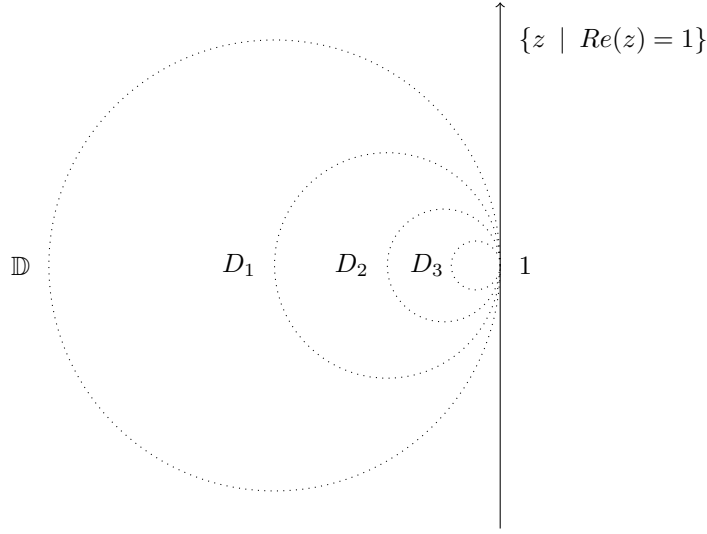
		\node[label={below: The disks $D_n = \phi_{a^n}(\mathbb{D})$ shrinking to the point 1.} ] at (0 ,-3.5) {};
	\end{tikzpicture}
\end{center}
Note that if $f$ is $\mathbf{EB}$ then $C_{\phi_{a^n}}f\in H^\infty$ for some $n\in\mathbb{N}$. Hence we immediately get the following result about minimal invariant subspaces of $C_{\phi_a}$.

\begin{prop}\label{bounded element}
	Let $M\in\mathrm{Lat}(C_{\phi_a})$. Then $M$ contains an $\mathbf{EB}$ function if and only if it contains an $H^\infty$ function. Additionally if $M$ is minimal, then $M=K_f$ for some $f\in H^\infty$.
\end{prop}



It is known that the operator $C_{\phi_a}$ is a hypercyclic operator, i.e, there exists $f \in H^2$ such that the orbit $(C_{\phi_a}^nf)_{n\geq0}$ is dense in $H^2$. Such $f$ are called hypercyclic vectors for $C_{\phi_a}$ and they form a dense subset of $H^2$ (see \cite{lineardynamics} or \cite{shapirocyclic}). The next result shows that hypercyclicity of $f$ and the eventual boundedness of its derivatives are incompatible.

\begin{prop}\label{hypercyclic vector}
	If any derivative of $f$ is $\mathbf{EB}$, then $f$ is not a hypercyclic vector.
\end{prop}
\begin{proof}
	Let $f^{(n)}$ be $\mathbf{EB}$ for some $n \geq 1$ and consider $e_n(z) = z^n$. Suppose on  the contrary that $f$ is a hypercyclic vector. Then there exists a subsequence $(C_{\phi_a}^{k_l}f)_{l \in \mathbb{N}}$ such that $C_{\phi_a}^{k_l}f \to e_n$ as $l \to \infty$. Note that for every $g \in H^2$ we have $\langle g, e_n \rangle = \frac{g^{(n)}(0)}{n!}$. So
	$$\langle C_{\phi_a}^{k_l}f, e_n \rangle = \frac{\left( C_{\phi_a}^{k_l}f \right)^{(n)}(0)}{n!} = \frac{(a^{k_l})^{n} f^{(n)} \circ \phi_{a^{k_l}}(0) }{n!} \xrightarrow[]{l \to \infty}0$$
	because $f^{(n)}$ is $\mathbf{EB}$ and $(a^{k_l})^{n} = (a^{n})^{k_l} \to 0$ as $l \to \infty$. On the other hand
	$$\langle C_{\phi_a}^{k_l}f, e_n \rangle \to \langle e_n, e_n\rangle = \|e_n\|^2.$$
	Thus $e_n = 0$ which is absurd. This contradiction establishes the result.
	\end{proof}

However, there exist cyclic vectors for $C_{\phi_a}$ (whose orbits are complete in $H^2$) with derivatives all $\mathbf{EB}$. Recall that for each $\alpha \in \mathbb{D}$, the reproducing kernel at $\alpha$ is defined by $\kappa_\alpha(z) =  \frac{1}{1 - \overline{\alpha}z}$ and each $\kappa_\alpha^{(n)}$ is $\mathbf{EB}$ for $n \in \mathbb{N}$ since $\kappa_\alpha$ is analytic on $S^1$. 

\begin{prop}\label{rparecyclic}
	Let $\kappa_\alpha \in H^2$ be a reproducing kernel. Then $\kappa_\alpha$ is a cyclic vector for $C_{\phi_a}$ if and only if $\alpha \neq 0$.
\end{prop}

\begin{proof}
	If $\kappa_\alpha$ is cyclic then $\alpha \neq 0$, otherwhise  $\kappa_\alpha = \kappa_0 = 1$ and $K_{\kappa_\alpha}$ is the space of constants functions. For the other direction, let $\kappa_\alpha$ with $\alpha \neq 0$. Note that $1 - \overline{\alpha} + \overline{\alpha}a \neq 0$, otherwise  $1 = \overline{\alpha}(1 - a)$ and this is not possible because $\overline{\alpha}, (1 - a) \in \mathbb{D}$. Thus
	$$\kappa_\alpha \circ \phi_a(z) = \frac{1}{1 - \overline{\alpha}(az +1 - a)} = \frac{1}{1 - \overline{\alpha} + \overline{\alpha}a - \overline{\alpha}az} = \frac{1}{1 - \overline{\alpha} + \overline{\alpha}a} \left( \frac{1}{1 - \frac{\overline{\alpha}az}{1 - \overline{\alpha} + \overline{\alpha}a}} \right).$$
	We observe that a function of the form $h(z) = \frac{1}{1 - \overline{y}z}$ where $y \in \mathbb{C}$ belongs to $H^2$ if, and only if, $y \in \mathbb{D}$; so $\frac{\overline{\alpha}a}{1 - \overline{\alpha} + \overline{\alpha}a} \in \mathbb{D}$. Consequently, for every $a \in (0,1)$ we have
	$$ \kappa_\alpha  \circ \phi_a =  \frac{1}{1 - \overline{\alpha} + \overline{\alpha}a} \kappa_{ \frac{\alpha a}{1 - {\alpha} + {\alpha}a}}.$$
	Now let $f \in H^2$ such that $\langle f,  \kappa_\alpha  \circ \phi_{a^n} \rangle = 0.$ Then $f( \frac{\alpha a^n}{1 - \alpha + \alpha a^n}) = 0$ for every $n \in \mathbb{N}$. As the sequence  $ \{ \frac{\alpha a^n}{1 - \alpha + \alpha a^n} \}_{n \in \mathbb{N}}$ is a sequence of distinct points (because $\alpha \neq 0$), $\frac{\alpha a^n}{1 - \alpha + \alpha a^n} \to 0$ and $f$ is analytic at $0$ we conclude that $f = 0$. Thus $K_{\kappa_\alpha } = H^2$.
\end{proof}

\section{Minimal invariant subspaces of $C_{\phi_a}$}

From now on, consider $a \in (0,1)$ fixed. Recall that the ISP has a positive solution if and only if every minimal invariant subspace of the operator $C_{\phi_a}$ is one-dimensional. Also every minimal invariant subspace is a cyclic susbspace of the form $K_f$ where $f \in H^2$ and 
$$K_f := \overline{\mathrm{span} \{f, C_{\phi_a}f, C^2_{\phi_a}f, \ldots \}}.$$
Indeed, if $M$ is a minimal invariant subspace of $C_{\phi_a}$ then $M = K_f$ for every $f \neq 0$ such that $f \in M$. In this section we will present a class of functions $f$ for which we know that the equivalence $$K_f \textrm{ is minimal} \Longleftrightarrow \mathrm{dim}  \ K_f = 1$$ is true. Our first main result is the following.

\begin{thm}\label{limit1}
	If $g \in H^2(\mathbb{D})$ with $\lim\limits_{z \to 1} g(z)=L\neq 0 $ ($z\to 1$ within $\mathbb{D}$), then $K_g$ contains the constants. In particular if $K_g$ is minimal, then $g\equiv L$ and $\mathrm{dim} \ K_g = 1$. 
\end{thm}

\begin{proof}
	There exists a $\delta > 0$ such that if $z \in \mathbb{D}$ and $|z - 1| < \delta$ then $|g(z) - L| < 1$. In particular, for every $z \in B(1,\delta) \cap \mathbb{D}$ we obtain
	
	$$|g(z)| \leq |g(z) - L)| + |L| < 1 + |L| =: K$$
	As $\phi_{a^n}(\mathbb{D}) = a^n \mathbb{D} + 1 -a^n$ are circles with center $1 - a^n$ (converging to $1$) and radius $a^n$ (converging to $0$) there exists $n_0 \in \mathbb{N}$ such that for all $n \geq n_0$ we have $\phi_{a^n}(\mathbb{D}) \subseteq B(1,\delta) \cap \mathbb{D}$. So, given $re^{i\theta} \in \mathbb{D}$ we conclude that $a^nre^{i\theta} + 1 - a^n \in B(1,\delta) \cap \mathbb{D}$ and so
	$$|g \circ \phi_{a^n}(re^{i\theta})|^2 = |g(a^nre^{i\theta} + 1 - a^n)|^2 \leq K^2 \,\ \,\ \forall n \geq n_0.$$
	Consequently we get
	\[
	\| C_{\phi_{a^n}}g \|^2 = \| g\circ \phi_{a^n} \|^2 = \sup \limits_{0 < r < 1} \frac{1}{2\pi}\int\limits_{0}^{2\pi} | g\circ \phi_{a^n} (re^{i\theta}) |^2 d\theta \leq K^2 \,\ \,\ \forall n \geq n_0.
	\]
	This shows that the sequence $ \left( C_{\phi_{a^n}}g \right)_{n \in \mathbb{N}}$ is bounded. Moreover, for each $z \in \mathbb{D}$ we have $a^nz + 1 - a^n \to 1$ as $n \to \infty$ and then
	$$  g\circ \phi_{a^n}(z) = g(a^nz + 1 - a^n) \to L.$$ 
	So $ \left( C_{\phi_{a^n}}g \right)_{n \in \mathbb{N}}$ converges pointwise to $L$. By (\cite{Cowen-Maccluer}, Corollary $1.3$) $ \left( C_{\phi_{a^n}}g \right)_{n \in \mathbb{N}}$ converges weakly to the constant function $L$. So $L$ belongs to the weak closure of the convex set $\mathrm{span}_{n\geq 0} (C_{\phi_a}^ng)$ which is equal to the norm closure by Mazur's Theorem. So $L \in K_g$ and hence $K_g = K_{L}$ since $L \neq 0$.
\end{proof}

Every function $g \in A\mathbb(\mathbb{D})$ (where $A(\mathbb{
	D})$ denotes the disk algebra) such that $g(1) \neq 0$ satisfies the above hypothesis. So for such functions $K_g$ is minimal if and only if $g$ is a constant. Notice that the theorem above does not cover polynomials $g$ that vanish at $1$. However the polynomial case in general is solved easily by the following remark.

\begin{rem}
	
	For every polynomial $p$ and $a \in (0,1)$, the function $p \circ \phi_a$ is again a polynomial with degree equal to that of $p$. In particular, $K_p \subseteq P(n)$ where $P(n)$ denotes the space of polynomials of degree at most $n=\mathrm{deg}(p)$. This implies that $dim \,\ K_p < \infty$. Thus $K_p$ is minimal precisely when $\mathrm{dim} \ K_p = 1$ since it will contain an eigenvector.

\end{rem}




Following Theorem \ref{limit1}, it is clear that the boundary behaviour of $g\in H^2$ at $1$ plays a key role in characterizing $K_g$. Hence we propose the following three cases:

\begin{itemize}
	\item[(A)] $g(1 - a^n)$ converges to a number $L \neq  0$.
	\item[(B)] $g(1 - a^n)$ converges to 0.
	\item[(C)] $g(1 - a^n)$ does not converges.
\end{itemize}
These three cases cover all possibilities. Recall that $g'$ is $\mathbf{EB}$ if $g'$ is bounded on $D_{n}$ for some $n\in \mathbb{N}$. We begin our analysis with a lemma.


\begin{lem}\label{lemmanv}
	If $g \in H^2$ is such that $g'$ is $\mathbf{EB}$, then 
	
	$$ \int_{\mathbb{D}}\left| {g'(w)} \right|^2N_{\phi_{a^{n}}}(w) d A(w) \to 0 \,\ \textrm{ as } n \to \infty.$$
\end{lem}

\begin{proof}
	
	Let $n_0  \in \mathbb{N}$ such that $g'$ is bounded in $D_{n_0}$. Applying Theorem \ref{CVF} for $f = |g'|^2$ and $\phi = \phi_{a^n}$ we obtain:
	\begin{align*}
		\int\limits_{\mathbb{D}} |g'(w)|^2N_{\phi_{a^{n}}}(w)dA(w) & =  \int\limits_{\mathbb{D}} |g'(a^nz +1 - a^n)|^2 |\phi_{a^n}'(z)|^2 \log{\frac{1}{|z|}}dA(z) \\
		& \leq  \int\limits_{\mathbb{D}} C a^{2n} \log{\frac{1}{|z|}}dA(z) = a^{2n}K  \,\ \,\ \,\ ( n \geq n_0 )
	\end{align*}  
	where $C$ is a constant such that $|g'(a^nz + 1 - a^n)|^2 \leq C$ for all $n \geq n_0$ and $K$ is the constant given by $C$ times the integral of $\log \frac{1}{|z|}$ over $\mathbb{D}$. So, if $n \to \infty$ then $a^{2n} \to 0$ and we conclude the proof.    
\end{proof}
For each $s \in \mathbb{C}$  Hurst showed in (\cite[Lemma 7]{Hurst}) that $f_s(z) = (1 - z)^s$ belongs to $H^2$ if and only if $\Re(s) > -\frac{1}{2}$. For $\Re(s) > -\frac{1}{2}$ these functions are eigenvectors 
$$C_{\phi_a}f_s(z) = f_s \circ \phi_a(z) = (1 - az - 1 + a)^s = a^sf_s(z).$$
We arrive at the first central result of this section.

\begin{thm}\label{newmain}
	Suppose that $f = f_sg$ for some $g \in H^2$ and $Re(s) \geq 0$. If $g'$ is $\mathbf{EB}$ and there exists a subsequence of $(g(1 - a^n))_{n \in \mathbb{N}}$ which is bounded away from zero, then $f_s \in K_f$. So $K_f$ is minimal if and only if $\dim K_f = 1$.

	\begin{proof}
		
		By hypothesis $f \in H^2$ (because $f_s \in H^{\infty}$) and we can choose a subsequence $(g(1 - a^{n_k}))_{k \in \mathbb{N}}$ which is bounded away from zero, i.e, there exists a constant $M$ such that $|g(1 - a^{n_k})| \geq M > 0$. As $f_s$ is an eigenvector we obtain the following relation
		$$C_{\phi_{a}}^{n_k} f = C_{\phi_{a}}^{n_k} f_sg = a^{{n_k}s}f_sC_{\phi_{a}}^{n_k} g.$$
		Using the expression above, Theorem \ref{NVF} and Lemma \ref{lemmanv} we obtain
		$$
		\begin{aligned}
			\left\|\frac{C_{\phi_{a}}^{n_k} f}{a^{n_ks}g(1 - a^{n_k})} - f_s \right\|_{2}^{2} & =  \left\|\frac{f_s C_{\phi_{a}}^{n_k}g}{g(1 - a^{n_k})} - f_s \right\|_{2}^{2}
			\leq \|f_s\|^2_{\infty} \left\|C_{\phi_{a^{n_k}}} \left( \frac{g}{g(1 - a^{n_k})}- 1 \right)\right\|_{2}^{2} \\
			&=2 \|f_s\|^2_{\infty} \int_{\mathbb{D}}\left| \frac{g'(w)}{g(1 - a^{n_k})}\right|^{2} N_{\phi_{a^{n_k}}}(w) d A(w)\\
			& \leq \frac{2 \|f_s\|^2_{\infty}}{M^2} \int_{\mathbb{D}}\left| {g'(w)} \right|^2N_{\phi_{a^{n_k}}}(w) d A(w) \xrightarrow{k \to \infty } 0 .
		\end{aligned}
		$$
		This means that $f_s$ is the norm-limit of a sequence of elements in $\mathrm{span} \{f, C_{\phi_a}f, \ldots \}$. Hence $f_s \in K_f$ and by minimality we have $K_f = K_{f_s}$ which is one dimensional because $f_s$ is an eigenvector.
	\end{proof}
\end{thm}

\begin{cor}\label{CORAC}
	Suppose that $g \in H^2$ is such that $g'$ is $\mathbf{EB}$ and there exists a subsequence of $(g(1 - a^n))_{n \in \mathbb{N}}$ which is bounded away from zero, then $1 \in K_g$. So $K_g$ is minimal if and only if $\mathrm{dim} \ K_g = 1$.
\end{cor}

\begin{proof}
	Apply the above theorem to $s = 0$.
\end{proof}

The result above corrects an error in \cite[Theorem 4.2]{noor} which made the conclusion incorrect. The hypothesis of $g'$ being $\mathbf{EB}$ is in fact necessary, as the following example demonstrates. 

Choose $g(z) = f_t(z) = (1 - z)^{ \frac{2 \pi i}{\log {a}}}$ where $t = \frac{2 \pi i}{ \log{a}}$ in Corollary \ref{CORAC}. Then
	$$ g( 1 - a^n) =  (a^n)^{ \frac{2 \pi i}{ \log{a} }} = e^{  \frac{2 \pi i}{\log {a}} \log{a^n}} = e^{2 \pi i n} = 1 $$
	for all $n\in\mathbb{N}$ and in particular $(g(1 - a^n))_{n \in \mathbb{N}}$ is bounded away from zero. But for $r \in (0,1)$
	$$|g'(r)| = \left| \frac{2 \pi i}{\log a} \frac{1}{r - 1} e^{\frac{2 \pi i}{\log a} \log (1 - r)} \right| = \left| \frac{2 \pi}{ \log a} \right| \left| \frac{1}{r - 1} \right| \xrightarrow{r \to 1}  \infty $$
	shows that $g'$ is not $\mathbf{EB}$. Then $K_g=\mathbb{C}g$ because $g$ is a $C_{\phi_a}$-eigenvector and $1\notin K_{g}$. 
	

As a consequence of the last corollary, we resolve cases $(A)$ and $(C)$:
\begin{thm}\label{AC}
	If $g \in H^2$ belongs to case $(A)$ or $(C)$ defined above and $g'$ is $\mathbf{EB}$, then $K_g$ is minimal if and only if $dim \,\ K_g = 1$. 
\end{thm}
\begin{proof} 	If $g$ belongs to case $(A)$ then $(g(1 - a^n))_{n \in \mathbb{N}}$ converges to a non-zero number, say $L$. Then, there exists $n_0 \in \mathbb{N}$ such that for every $n \geq n_0$, $|g(1 - a^n) - L| < \frac{|L|}{2}$ and then $|g(1 - a^n)| \geq \frac{|L|}{2} > 0$ for $n \geq n_0$. If $g$ belongs to case $(C)$ then the radial limit does not exist, and in particular it is not $0$. So we can find an $\epsilon > 0$ and a subsequence $(g(1 - a^{n_k}) )_{k \in \mathbb{N}}$ such that $|g(1 - a^{n_k})| \geq \epsilon > 0$. The result follows by Corollary \ref{CORAC}.
\end{proof}

Case $(B)$ can also be solved assuming $g$ is analytic at $1$. 

\begin{thm}\label{analyticat1}
		Suppose that $g \in H^2$ is analytic at $1$. Then $K_g$ is minimal if and only if $\mathrm{dim} \ K_g = 1$.
\end{thm}
\begin{proof}
		Notice that as $g$ is analytic at $1$, then so is $g'$. In particular $g'$ is $\mathbf{EB}$. If $g(1) \neq 0$ then case $(A)$ of the previous theorem gives the result.  Suppose that $g(1) = 0$, then $g(z) = (1 - z)^K h(z)$ on some neighborhood $V$ of $1$, where $K$ is the multiplicity of the zero $1$, $h$ is analytic at $V$ and $h(1) \neq 0$. Choose $n_0 \in \mathbb{N}$ such that for all $n \geq n_0$, $D_{n} \subset V$. Note that $h'$ is $\mathbf{EB}$ because $h'$ is also analytic at $V$. Moreover,
		$$g \circ \phi_{a^n}(z) = a^{nK}(1 - z)^K h \circ \phi_{a^n}(z).$$
		So $h \circ \phi_{a^{n}} \in H^{\infty}$ because $h$ is bounded $D_n$ (in particular, $h \circ \phi_{a^{n}} \in H^{2}$). As the function $f = f_K  h \circ \phi_{a^{n_0}} = \frac{ g \circ \phi_{a^n} }{a^{nK}} $ satisfies the hypothesis of Theorem \ref{newmain}, we conclude that $f_K \in K_{f_K h \circ \phi_{a^{n_0}}} = K_{ g \circ \phi_{a^n} } \subseteq K_g$. Thus $K_g = K_{f_K}$ by minimality and we are done.
	\end{proof}
	
	We conclude this section by considering case $(B)$ more carefully. Even with the $\mathbf{EB}$ hypothesis over $g'$, case $(B)$ appears to be delicate. However if $g(1 - a^n)$ has a subsequence that converges to $0$ at a \emph{sufficiently} slow rate, then we can still obtain a positive result.


	\begin{thm}\label{Slow decay case B}
		Suppose that $g \in H^2$ is such that $g'$ is $\mathbf{EB}$ and there exists $0<\epsilon<1$ and a constant $L > 0$ such that $|g(1 - a^{n_k})| \geq L a^{n_k(1 - \epsilon) }$ for some subsequence $(n_k)_{k \in \mathbb{N}}$. Then $1 \in K_g$, and $K_g$ is minimal if and only if $\mathrm{dim} \ K_g = 1$.
	\end{thm}
	
	\begin{proof} Using similar computations as in the proof of Theorem \ref{newmain}, we have
		$$
		\begin{aligned}
			\left\|\frac{C_{\phi_{a}}^{n_k} g}{g(1 - a^{n_k})} - 1 \right\|_{2}^{2}  
			& = 2  \int_{\mathbb{D}}\left| \frac{g'(w)}{g(1 - a^{n_k})}\right|^{2} N_{\phi_{a^{n_k}}}(w) d A(w)+\left| \frac{g \left(1-a^{n_k}\right)}{g \left(1-a^{n_k}\right) }- 1 \right|^{2} \\
			& =  2 \int_{\mathbb{D}} \frac{ |g'(w)|^2}{ |g(1 - a^{n_k})|^2} N_{\phi_{a^{n_k}}}(w) d A(w) \\
			&= 2 \int\limits_{\mathbb{D}} \frac{|g'(a^{n_k}w +1 - a^{n_k})|^2}{|g(1 - a^{n_k})|^2} |\phi_{a^{n_k}}'(w)|^2 \log{\frac{1}{|w|}}dA(w) \\
			& \leq 2 \int\limits_{\mathbb{D}} C^2  \frac{1}{L^2 a^{2n_k} a^{- 2 n_k \epsilon }} a^{2n_k }\log{\frac{1}{|w|}}dA(w) \\
			& = a^{2n_k \epsilon} M \xrightarrow{k \to \infty} 0
		\end{aligned}$$
		where $C$ is a upper bound for the values of $g'$ in some open ball $D_{n_{k_0}}$ and $M$ is a constant. This proves the result.
	\end{proof}
	
	To obtain a complete solution for case $(B)$ under the $\mathbf{EB}$ hypothesis over $g'$, it is natural to ask what happens if $g(1 - a^n)\to 0$ faster than required by Theorem \ref{Slow decay case B}. For instance if
	$$|g(1 - a^n) | \leq a^{\frac{n}{2}}$$
	then the next result shows that the series $ \sum\limits_{n = 1}^{\infty}C_{\phi_{a^n}}g$ converges in $H^2$. 
	
	\begin{prop}\label{Abs convergence hypothesis}
		Let $g \in H^2$ with $g'$ an $\mathbf{EB}$ function. Then $h:= \sum\limits_{n = 1}^{\infty}C_{\phi_{a^n}}g\in H^2$ if and only if $\sum\limits_{n = 1}^{\infty}|g(1 - a^n)|$ converges in $\mathbb{C}$. Moreover $h\in K_g$.
	\end{prop}
	\begin{proof}
		
		Suppose that $\sum\limits_{n = 1}^{\infty}|g(1 - a^n)|$ converges in $\mathbb{C}$. Using Theorems \ref{NVF} and \ref{CVF} we conclude that for all $n \geq n_0$ (where $D_{n_0}$ is a ball in which $g'$ is bounded)
		$$\begin{aligned}
			\|C_{\phi_{a^n}} g\|^2 & = \int\limits_{\mathbb{D}} |g'(w)|^2N_{\phi_{a^{n}}}(w)dA(w) + |g(1 - a^n)|^2 \\
			& =   \int\limits_{\mathbb{D}} |g'(a^nz +1 - a^n)|^2 |\phi_{a^n}'(z)|^2 \log{\frac{1}{|z|}}dA(z)  + |g(1 - a^n)|^2\\
			& \leq a^{2n} L +  |g(1 - a^n)|^2 
		\end{aligned}$$
		where $L$ is a positive constant. So
		$$\|C_{\phi_{a^n}} g\| \leq \sqrt{a^{2n} L + |g(1 - a^n)|^2} \leq a^{n} \sqrt{L} + |g(1 - a^n)| \,\ \,\ \textrm{ for } n \geq n_0.  $$
		Since $a < 1$ and $\sum\limits_{n = 1}^{\infty}|g(1 - a^n)| < \infty$, the comparison test implies
		$$\sum\limits_{n = 1}^{\infty} \|  C_{\phi_{a^n}}g \| < \infty.$$
		The reciprocal follows from
		$$\|C_{\phi_{a^n}} g\|^2 = \int\limits_{\mathbb{D}} |g'(w)|^2N_{\phi_{a^{n}}}(w)dA(w) + |g(1 - a^n)|^2 \geq |g(1 - a^n)|^2$$
		and by the comparison test again.
	\end{proof}
	Under the hypothesis of Proposition \ref{Abs convergence hypothesis} if we define $h_k:=\sum_{n=k}^\infty C_{\phi_{a^n}}g$, then $h\in H^2$ obviously implies all $h_k\in H^2$ for $k\geq 1$. In this case, if $K_g$ is a minimal invariant subspace for $C_{\phi_a}$ then we must have $K_g=K_{h_k}$ for all $k\in\mathbb{N}$. Notice that if some $h_l$ were independent of the rest (say $h_l\notin\overline{\mathrm{span}}(h_k)_{k> l}$), then $K_{h_l}$ would properly contain $K_{h_k}$ for $k>l$ and in particular $K_g$ cannot be minimal.   This leads us to conjecture that \\ \\ \emph{If $g \in H^2$ and $\sum\limits_{n = 1}^{\infty}C_{\phi_{a^n}}g\in H^2$, then $K_g$ is minimal if and only if $g$ is a $C_{\phi_a}$-eigenvector.}


\section{The role of the zero set }

In this section we present some results showing how the cardinality of the zero set of $f$ effects the dimension of $K_f$. We will use the notation $Z(f)$ to denote the set of zeros of $f$ in $\mathbb{D}$ and $|Z(f)|$ its cardinality.

\begin{prop}\label{numberof0}
	If $0 < |Z(f)| < \infty$ then $dim  \,\ K_f \geq 2$.
\end{prop}

\begin{proof}
	By the hypothesis of finite zeros, we can choose $0 < K < 1$ such that $f$ is zero-free in the annulus $\mathbb{D} - \overline{B(0,K)}$. Moreover, we can choose $n_0 \in \mathbb{N}$ such that for every $n \geq n_0$, $f \circ \phi_{a^n}(\mathbb{D}) \subseteq \mathbb{D} - \overline{ B(0,K) }$.
	
	\begin{center}
		\begin{tikzpicture}
			\draw[dotted] (0,0) circle [radius=3cm];
			\draw (0,0) circle [radius=2.0cm];
			\draw[dotted](2.675, 0) circle [radius= 0.325cm];
			\draw(0,0) -- (0,2);
			\node[label={left: $\mathbb{D}$} ] at (-3,0) {};
			\node[label={left: $K$} ] at (0,1) {};
			\node[label={below: $0$} ] at (0,0) {};
			\node[label={below: $\overline{B(0,K)}$} ] at (0,3) {};
			\node[label={right: $f \circ \phi_{a^{n_0}}
				(\mathbb{D})$} ] at (3,0) {}; 
			\node[label={below: Choosing $K$ and $n_0$.} ] at (0 ,-3) {};
		\end{tikzpicture}
	\end{center}

	If $n \geq n_0$ we claim that $\{f, f \circ \phi_{a^n} \}$ is a L.I set. In fact, consider any scalars $\alpha, \beta \in \mathbb{C}$ such that $\alpha f + \beta f \circ \phi_{a^n} = 0$. If $z_0 \in Z(f)$ we have
	
	$$\beta f \circ \phi_{a^n}(z_0) = \alpha f(z_0) + \beta f \circ \phi_{a^n}(z_0) = 0$$
	and since $f \circ \phi_{a^n}$ is zero-free, we conclude that $\beta = 0$ and, consequently, $\alpha = 0$.
\end{proof}

This result has the following interesting consequence: If the ISP is true, then every $K_f$ with $f$ satisfying $0 < |Z(f)| < \infty$ is necessarily non-minimal. On the other hand if $K_f$ is minimal for such an $f$, then the ISP is false and $K_f$ would furnish a counterexample. We note that if $|Z(f)| = 0$ or $|Z(f)| = \infty$, then $K_f$ may or may not be minimal. For instance, consider $g_1(z) = z^2 + 1$ and $g_2(z) = 1 - z$. Both functions are zero-free in the disk, but $K_{g_2}$ is minimal because $g_2$ is an eigenvector whereas $K_{g_1}$ is not minimal by Theorem \ref{limit1}. For the $|Z(f)| = \infty$ case, we first consider the following result.

\begin{prop}
	Let $w \in \mathbb{D}$. If $f(w) \neq 0$ but there exists $n_0 \in \mathbb{N}$ such that for all $n \geq n_0$ we have $f(a^nw + 1 - a^n) = 0$ then $K_f$ is not minimal.
\end{prop}

\begin{proof}
	In fact, if $K_f$ is minimal then $K_f = K_{f \circ \phi_{a^{n_0}}}$ and consequently $(K_f)^{\perp} = (K_{f \circ \phi_{a^{n_0}}})^{\perp}$. Considering the reproducing kernel $\kappa_w$ we observe that $\kappa_w \in (K_{f \circ \phi_{a^{n_0}}})^{\perp}$ since $f(a^nw + 1 - a^n) = 0$ ($n \geq n_0$) but $\kappa_w \notin (K_f)^{\perp}$ because $f(w) \neq 0$. This is a contradiction.
\end{proof}
Consider the Blaschke sequence $\{\phi_{a^n}(0)\}_{n \geq 2} = (1 - a^n)_{n \geq 2}$. The corresponding Blaschke product
	$$B(z) = \prod\limits_{n = 2}^{\infty} \left( \frac{1 - a^n - z }{ 1 - \overline{( 1 - a^n)}z } \right).$$ 
	has infinitely many zeros. Note that $B(0) \neq 0$ but $B(a^n 0 + 1 - a^n) = B(1 - a^n) = 0$ for $n \geq 2$, so by the previous result $K_B$ is not minimal. The only remaining case is when $f$ has infinitely many zeros and yet $K_f$ is minimal. 
	
	Let $a = \frac{1}{2}$. Consider the function $f = e_0 + f_s$ where $s = \frac{2 \pi i }{\log a}$ and $e_0$ is the constant function $1$. Note that $e_0 + f_s$ is an $C_{\phi_{a}}$-eigenvector:
	\[
	C_{\phi_a}(e_0  + f_s) = e_0  + a^sf_s = a^{\frac{2 \pi i}{\log a}}f_s = e_0  + f_s.
	\]
	So $K_f$ is minimal. Moreover, note that
	\[
	(e_0  + f_s)(1 - \sqrt{2}) = 1 + f_s(1 - \sqrt{2}) = 1 + e^{ \frac{2 \pi i }{- 2 \log \sqrt{2}} \log \sqrt{2}  } = 1 + e^{-\pi i} = 0
	\]
	Now, consider the sequence $\{\phi_{a^n}(1 - \sqrt{2})\}_{n \in \mathbb{N}} \subseteq \mathbb{D}$. Then
	\begin{align*}
		(e_0  + f_s)(\phi_{a^n}(1 - \sqrt{2})) = C_{\phi_{a}}^n(e_0  + f_s)(1 - \sqrt{2}) =(e_0 + f_s)(1 - \sqrt{2}) = 0
	\end{align*}
	and thus $f$ has infinitely many zeros in $\mathbb{D}$. The final result of this section shows that we can always find a function in $K_f$ that is orthognonal to $f$.

\begin{prop}
	Suppose that $0 < |Z(f)| < \infty$ and let $z_0 \in Z(f)$. There exists a non-zero $g \in K_f$ such that $\langle g,h \rangle = 0$ for every $h \in K_f$ such that $h(z_0) = 0$.
\end{prop}
\begin{proof}
	Let $z_0 \in \mathbb{D}$ a zero of $f$ and consider the continuous map $E_{z_0}: K_f \to \mathbb{C}$ which is exactly the restriction of the evaluation map at $z_0$ defined in $H^2$. Note that $E_{z_0}$ is surjective because there exists $n \geq n_0$ such that $C_{\phi_a}^{n}f(z_0) \neq 0$. So,
	
	$$K_f \ominus Ker_{E_{z_0}} \simeq  \dfrac{K_f}{Ker_{E_{z_0}}} \simeq \mathbb{C}.$$
	which implies that $K_f \ominus Ker_{E_{z_0}}$ is a one-dimensional space. Let $g \in K_f \ominus Ker_{E_{z_0}}$ a non-null element. Thus $\langle g,h \rangle = 0$ whenever $h \in K_f$ is such that $h(z_0) = 0$.
	\end{proof}

\section{Cyclicity and universality}

We end this article with a result that highlights a connection between cyclicity and universality. We note that the best known examples of universal operators are adjoints of analytic Toeplitz operators $T_\phi f=\phi f$ for $\phi\in H^\infty$, $f\in H^2$ or those that are similar to them (see \cite{Cowen- Gallardo 1},\cite{Cowen- Gallardo 2},\cite{Cowen- Gallardo 3}), and all such coanalytic Toeplitz operators are cyclic (see \cite{Wogen}).

\begin{thm}\label{newcar}
	If $T$ is a closed range cyclic operator with infinite dimensional kernel on a Hilbert space $H$, then $T$ is universal.
\end{thm}

\begin{proof}
	
	By the Caradus-Pozzi criterion (Theorem \ref{Caradus-Pozzi}) it is enough to prove that $T(H)$ has finite codimension. In fact for any cyclic $T$ the dimension of $T(H)^\perp$ is either $0$ or $1$. Although this may be known to experts, we provide a proof for the sake of completeness. Let $f$ be a cyclic vector for $T \in B(H)$ and define $N := \overline{\mathrm{span}_{n\geq 1}\{T^nf\}}$. Let $P: H \to N^{\perp}$ be the orthogonal projection onto $N^{\perp}$. Now let $g$ be any element in $T(H)^{\perp}$. Then we have $\langle g, T^nf \rangle = 0$ for $n\geq 1$ and consequently $g \in N^{\perp}$. 
	Since $f$ is a cyclic vector, we can find a sequence $g_n \in \overline{\mathrm{span}_{n\geq 0}\{T^nf\}}$ such that $g_n \to g$. If we write $g_n = \alpha_n f + t_n$ where the $\alpha_n$ are scalars and $t_n \in N$, then we obtain $\alpha_n f + t_n \to g$. Applying $P$ to this we conclude that $\alpha_nPf \to Pg=  g$ and hence $g=\alpha Pf$ for some $\alpha\in\mathbb{C}$. Since $g$ is an arbitrary element of $T(H)^{\perp}$, the latter space is at most one-dimensional. Finally $\mathrm{codim} \ T(H) =\mathrm{dim} \ T(H)^\perp\leq 1$ because $T$ has closed range and the result follows. 
\end{proof}

Since all known examples of universal operators obtained via the Caradus criterion have closed range, it leads one to suspect whether an operator having closed range is necessary or sufficient (assuming $\infty$-dimensional kernel and range) for universality. But the following counterexamples generously provided by Professor Jonathon Partington prove otherwise. Let $U$ be a universal operator and $V$ an arbitrary operator without closed range (e.g. multiplication by $z-1$ on $H^2$). Then $U \oplus V$ is universal on $H\oplus H$, but does not have closed range. And conversely, consider the operator defined on $H \times H$ by $T(x,y)=(0,x)$ for $x\in H$. Then $T$ has closed infinite dimensional kernel and range,  both equal to $0 \times H$.  Therefore $T^2=0$, so it cannot be universal.

Finally, we note that Theorem \ref{newcar} may simplify the proofs of some known results.  Recall that if $\phi \in LFT(\mathbb{D})$ is hyperbolic then $C_{\phi} - \lambda$ is universal for all eigenvalues $\lambda$. The classical proof of the automorphic case (see \cite{Nordgren-Rosenthal-Wintrobe}) and that of the recent non-automorphic case (see \cite{noor}) are both elaborate and involved. An elegant proof of the automorphic case was found recently by Cowen and Gallardo-Gutiérrez \cite{Cowen- Gallardo 3}. Theorem \ref{newcar} suggests an approach based on showing that the range of $C_{\phi} - \lambda$ is closed for some eigenvalue $\lambda$, rather than proving surjectivity. That is because when $\phi$ is hyperbolic, then $C_{\phi}$ is hypercyclic and in particular cyclic (see \cite[Theorem 1.47]{lineardynamics}). Also for every $\lambda \in \mathbb{C}$ and bounded linear operator $T$ we have
$$\mathrm{span} \{f, Tf, T^2f, \ldots \} = \mathrm{span} \{f, (T - \lambda)f, (T - \lambda)^2f, \ldots \}.$$
It follows that $T$ is cyclic if and only if $T-\lambda$ is cyclic. Therefore $C_{\phi} - \lambda$ is cyclic for all $\lambda\in\mathbb{C}$  if $\phi$ is hyperbolic. Moreover for $\lambda$ in the point spectrum of $\phi$, the kernel of $C_{\phi} - \lambda$ is infinite dimensional (\cite{Cowen-Maccluer}, Lemma $7.24$ and Theorem $7.4$). Therefore the closure of the range of $C_{\phi} - \lambda$ gives universality by Theorem \ref{newcar}.

\section*{Acknowledgements}
This work constitutes a part of the doctoral thesis of the second author, partially supported by the Conselho Nacional de Desenvolvimento Cient\'{i}fico e Tecnol\'{o}gico - CNPq Brasil, under the supervision of the third named author.

\bibliographystyle{amsplain}

\begin{thebibliography}{00}
	
	\bibitem{lineardynamics} F. Bayart and E. Matheron, Dynamics of linear operators. Cambridge Tracts in Mathematics 179, Cambridge University Press  (2009).
	
	\bibitem{shapirocyclic} P. Bourdon and J. H. Shapiro, Cyclic phenomena for composition operators. Mem. Amer. Math. Soc., 596 (1997).
	\bibitem{caradus} S. R. Caradus, Universal operators and invariant subspaces, Proc. Amer. Math. Soc. 23 (1969), 526-527. 
	
	\bibitem{noor} J. R. Carmo and S. W. Noor, Universal composition operators. J. Operator Theory (2022), 137-156.
	

	
	\bibitem{modernisp} I. Chalendar and J. R. Partington, Modern approaches to the Invariant Subspace Problem, Cambridge University Press, 2011.
	
	\bibitem{Chkliar} V. Chkliar, Eigenfunctions of the hyperbolic composition operator, Integral Equations Operator Theory (3) (1997) 364-367.
	
	\bibitem{Cowen-Maccluer} C. C. Cowen and B. MacCluer, Composition Operator on Spaces of Analytic Functions . Studies in Advanced Mathematics. CRC Press, Boca Raton, 1995.
	

	\bibitem{Cowen- Gallardo 1} Carl C. Cowen and Eva A. Gallardo Gutiérrez, Consequences of Universality Among Toeplitz Operators, J. Math. Anal. Appl. 432(2015), 484–503.
	
	\bibitem{Cowen- Gallardo 2} Carl C. Cowen and Eva A. Gallardo Gutiérrez, Rota’s universal operators and invariant subspaces in Hilbert
	spaces, J. Funct. Anal. 271(2016), 1130–1149
	
	\bibitem{Cowen- Gallardo 3} Carl C. Cowen and Eva A. Gallardo Gutiérrez, A new proof of a Nordgren, Rosenthal and Wintrobe theorem on universal operators. Problems and recent methods in operator theory, 97–102, Contemp. Math., 687, Amer. Math. Soc., Providence, RI, 2017.
	
	\bibitem{Gallardo-Gorkin} E. A. Gallardo-Guti\'{e}rrez, P. Gorkin, Minimal invariant subspaces for composition operators, J. Math. Pures Appl. 95 (2011) 245-259.
	
	
	\bibitem{Hurst} P.R. Hurst, Relating composition operators on different weighted Hardy spaces, Arch. Math. 68 (1997) 503-513.
	
	
	\bibitem{Matache minimal} V. Matache, On the minimal invariant subspaces of the hyperbolic composition operator, Proc. Amer. Math. Soc. 119 (3) (1993) 837-841.
	
	\bibitem{Matache eigenfunction} V. Matache, The eigenfunctions of a certain composition operator, Contemp. Math. vol. 213 (1998) 121-136.
	
	
	
	\bibitem{Mortini} R. Mortini, Cyclic subspaces and eigenvectors of the hyperbolic composition operator, Sém. Math. Luxembourg, in: Travaux Mathématiques, Fasc. VII, Centre Univ. Luxembourg, Luxembourg, 1995, pp. 69–79.
	
	\bibitem{Nordgren-Rosenthal-Wintrobe} E. Nordgren, P. Rosenthal, F.S. Wintrobe, Invertible composition operators on $H^p$, J. Funct. Anal. 73 (1987), 324-344.
	
	
	\bibitem{pozzi} E. Pozzi, Universality of weighted composition operators on $L^2([0,1])$ and Sobolev spaces. Acta Sci. Math., 78 (2012) 609-642.
	
	\bibitem{rota} G. C. Rota, On models for linear operators, Comm. Pure Appl. Math. 13 (1960), 469-472.

	
	\bibitem{Shapiro Annals} J. H. Shapiro, The essential norm of a composition operator. Ann. of Math. 125, 375-404 (1987).
	
	\bibitem{Shapiro Text} J. H. Shapiro, Composition operators and classical function theory. Universitext, Springer-Verlag (1993).
	
	

	
	\bibitem{Wogen} W. R. Wogen. On some operators with cyclic vectors. Indiana Univ. Math. J., 27 (1978) 163-171.
\end{thebibliography}

\end{document}